\documentclass[a4paper]{amsart}
\usepackage{amsmath}\usepackage{amsthm}
\usepackage{latexsym}
\usepackage[psamsfonts]{amssymb}
\usepackage[colorlinks=true,linkcolor=black,citecolor=black]{hyperref}
\usepackage{enumitem}
\newtheorem{thm}{Theorem} %only for theorems in introduction
\newtheorem{theorem}{Theorem}[section]
\newtheorem{prop}[theorem]{Proposition}
\newtheorem{lemma}[theorem]{Lemma}
\newtheorem{cor}[theorem]{Corollary}

\newtheorem*{lemma*}{Lemma}
\theoremstyle{definition}
\newtheorem{hypo}[theorem]{Assumption}%\renewcommand{\thehypo}{\Alph{hypo}}
\newtheorem*{hypo*}{Assumption}
\newtheorem{remark}[theorem]{Remark}
\newtheorem*{example*}{Example}
\newtheorem*{notation*}{Notation}
\newtheorem*{remark*}{Remark}

\numberwithin{equation}{section}

\newcommand{\comment}[1]{}
\newcommand{\la}{\,\langle\,}
\newcommand{\ra}{\,\rangle\,}
\newcommand{\lla}{\,\langle\!\langle\,}
\newcommand{\rra}{\,\rangle\!\rangle\,}
\newcommand{\ov}[1]{\overline{#1}}

\def\ind{\mathrm{Ind}}
\def\im{\operatorname{Im}}  \def\tr{\operatorname{Tr}}
\def\ker{\operatorname{Ker}} \def\stab{\operatorname{Stab }}
\def\sgn{\operatorname{sgn}}  
\def\sym{\operatorname{Sym}}\def\Symb{\operatorname{Symb}}
 \def\Q{\mathbb{Q}}\def\R{\mathbb{R}}\def\Z{\mathbb{Z}}\def\C{\mathbb{C}}
\def\le{\leqslant} \def\ge{\geqslant}
\def\SL{\mathrm{SL}} \def\PSL{\mathrm{PSL}}\def\GL{\mathrm{GL}}\def\PGL{\mathrm{PGL}}
 \def\V{\mathcal{V}}\def\W{\mathcal{W}}\def\H{\mathcal{H}}

\def\Sc{\mathcal{S}}  \def\M{\mathcal{M}}\def\RR{\mathcal{R}}
\def\X{X}
  \def\cc{\mathcal{C}}
\def\e{\varepsilon} \def\DD{\Delta} \def\G{\Gamma}\def\SS{\Sigma}\def\D{\mathcal{D}}
\def\dd{\delta} \def\ss{\sigma}
\def\wT{\widetilde{T}}\def\wc{\widetilde{\chi}}\def\wV{\widetilde{V}}
\def\g{\gamma}

\def\+{\,+\,}   %\def\={\;=\;}

\def\sm#1#2#3#4{\left(\begin{smallmatrix}#1&#2 \\ #3 & #4 \end{smallmatrix}\right)}

\def\be{\begin{equation}}  \def\ee{\end{equation}}

\def\vp{\varphi}\def\wg{\widetilde{\Gamma}}
\def\bsh{\backslash}
\def\rar{\rightarrow}

\title[Trace formula for Hecke operators]{On the trace formula for Hecke operators\\ on congruence subgroups}

\author{Alexandru A. Popa}
\address{Institute of Mathematics ``Simion Stoilow" of the Romanian Academy,
P.O. Box 1-764, RO-014700 Bucharest, Romania}
\address{E-mail: aapopa@gmail.com}

\begin{document}

\maketitle
\begin{center}
\emph{Dedicated to Don Zagier on the occasion of his 65th birthday.}
\end{center}

\begin{abstract}
We give a new, simple proof of the trace formula for Hecke operators on 
modular forms for finite index subgroups of the modular group. The proof 
uses algebraic properties of certain universal Hecke operators acting on 
period polynomials of modular forms, and it generalizes an approach developed by
Don Zagier and the author for the modular group. This approach leads to a 
very simple formula for the trace on the space of cusp forms plus the trace on the 
space of modular forms. As applications, we investigate what happens when
varying the weight or the level in the trace formula. 
\end{abstract}

%\tableofcontents

\section{Introduction and statement of results}\label{sec1}

Let $\G$ be a finite index subgroup of $\G_1=\SL_2(\Z)$, and let $\chi$ 
be a character of $\G$ with kernel of finite index in~$\G$. We denote by
$M_k(\G,\chi)$, $S_k(\G,\chi)$ the spaces of 
modular forms, respectively cusp forms for~$\G$ of weight~$k\ge 2$ 
and Nebentypus $\chi$. For $\SS$ a double coset of $\G$ inside 
its commensurator, we denote by~$[\SS]$ the 
associated operator acting on modular forms. In this paper, we give 
a simple formula for the combination of traces
  \be\label{0} \tr([\SS], M_k(\G,\chi)+S_k^c(\G,\chi))\, := \, 
  \tr([\SS], M_k(\G,\chi))\+\tr([\SS], S_k^c(\G,\chi))\;, \ee
under the assumption $|\G\backslash\SS|=|\G_1\backslash\G_1\SS|$, 
where $S_k^c(\G,\chi)$ denotes the space of anti-holomorphic cusp forms.

The proof is entirely algebraic, and it 
is based on the fact that~\eqref{0} is the trace of a universal Hecke element 
acting on the space of (vector) 
period polynomials associated to modular forms. For the full modular group, a 
method for computing the trace of this Hecke operator was 
sketched by Don Zagier more than 20 years ago~\cite{Z}. We sharpen 
this approach in an upcoming joint work~\cite{PZ}, whose main result, 
Theorem~\ref{T2.1} below, we take for granted in this paper. 
It is surprising that the same Hecke element--which is independent of 
the weight, congruence subgroup, Nebentypus 
and double coset--is the key to proving the trace formula for an arbitrary congruence
subgroup.

Compared to the full modular group case, in this paper we take a novel 
point of view and obtain a general trace formula for double coset operators 
acting on the period subspace of  
all~$\G_1$-modules~$\V$ that admit a $\G_1$-invariant, nondegenerate pairing 
(Theorem~\ref{Tcoh}). The period subspace is closely related to 
the parabolic cohomology group $H_P^1(\G_1, \V)$, and 
our results are then obtained by taking~$\V$ to be the module
induced from the $\G$-module~$\sym^{k-2} \C^2$, twisted by $\chi$, and using the 
Eichler-Shimura isomorphism and the Shapiro lemma, which are reviewed 
in Section~\ref{s2.0}. We also use the theory of period polynomials for 
finite index subgroups developed together with Vicen\c{t}iu Pa\c{s}ol 
in~\cite{PP}, and formulated in a more general context in Sections~\ref{s2} 
and~\ref{s2.0} below. Modulo this background material, the proof of the 
trace formula for modular forms is an immediate application of the cohomological 
trace formula in Theorem~\ref{Tcoh}, and it is given in Section~\ref{s2.1}. 

In a special case, Theorem~\ref{Tcoh} can be stated as an ``Euler-Poincar\'e trace 
formula'' that may hold in much greater generality than proved here. Taking $V$ to 
be the $\G$-module $\sym^{k-2} \C^2$ twisted by $\chi$ and using the Shapiro
lemma, Theorem~\ref{Tcoh} yields: 
\be\label{1}
\sum_{i} (-1)^{i+1}\tr([\SS], H^i(\G,V)) \,=\,\sum_{X\subset \ov{\SS}} \e_\G(X)\cdot \tr(M_X, V)\,,
\ee
where the sum on the right is over conjugacy classes $X$ in the projectivization 
$\ov{\SS}$ of $\SS$ with representatives $M_X\in \SS$, and 
$\e_\G(X)$ are simple conjugacy class invariants defined in~\eqref{1.50} below.  
The cohomology groups $H^i(\G,V)$ are nontrivial for $i=0,1$ in our situation
since $\G$ has cusps, but the formula makes sense for arbitrary Fuchsian groups
of the first kind and modules $V$, and even for higher rank groups for 
appropriately defined coefficients $\e_\G(X)$. This raises the question whether a cohomological 
approach exists in much greater generality than proved in this paper, leading 
to the trace formula~\eqref{1}. We hope to return to this 
question in future work. 

The resulting trace formula can be easily applied to investigate what happens 
when varying the weight or the level--see Theorems~\ref{T2} and~\ref{T3}. We give
here two examples.  Let $\SS=\G$ be the trivial double coset, and  
assume that~$\G$ is a finite index subgroup of $\G_1$ with no elliptic elements. 
Only the conjugacy class of the identity contributes 
in formula~\eqref{TFF}, and we obtain immediately the dimension formula:
\[\dim M_k(\G,\chi)+\dim S_k(\G,\chi) = \frac{k-1}{6}\cdot [\ov{\G}_1: \ov{\G}]+
\delta_{k,2},\]
assuming $\chi(-1)=(-1)^k$ if $-I\in \G$. For example if $\G=\G_0(11)$ and $k\ne 2$,
the sum of dimensions 
above equals $2(k-1)$ for all 5 characters $\chi$ with $\chi(-1)=(-1)^k$. 
The dimension formula for the cuspidal subspace also follows 
from the Riemann-Roch theorem~\cite{Sh}, but we obtain the formula above directly, 
without needing to compute the number of cusps and the genus of the associated
modular surface. This is consistent with one theme of this paper, that
the simplest formulas hold for the linear combination of traces~\eqref{0}.

As a generalization of the dimension formula, let
$T_\ell$ be the Hecke operator for a prime $\ell$ for the principal congruence subgroup $\G_n$ of 
level $n$. Then Theorem~\ref{T3}  immediately gives:
\be\label{3}
\lim_{\substack{n\rar \infty\\(n,\ell(\ell-1))=1 }}
\frac{\tr (T_\ell, M_k(\G_n)+S_k(\G_n) )}{\varphi(n)}=\frac{\ell^{k-1}-1}2.
\ee
A similar limit for $\G_0(N)$ in the level and weight aspect has been 
used by Serre to determine the distribution of the eigenvalue of a 
fixed $T_\ell$ of Hecke eigenforms of varying weight and level~\cite{S}. 
It also answers a question posed by Florin R\u{a}dulescu, who considers similar 
limits over a sequence of subgroups approaching the identity in ~\cite[Cor. 4]{R}. 
Similar limits for $\tr(T_n, S_k(\G_1(N)))$ as $N\rar \infty$ are computed in the 
sequel to this paper~\cite{P}, and we find that they are independent of 
$n$ and $k$ as long as $(N,n-1)$. 

Our approach for proving the trace formula is related to the theory of modular symbols. 
The Hecke operators acting on period polynomials are adjoints of the Hecke operators 
on modular symbols introduced by Merel~\cite{Me}, and our approach may also be interpreted 
as computing the trace of Hecke operators on the space of modular symbols. 

One could also apply our method to the module $\sym^{k-2} \C^2\otimes \Psi$, 
for a finite dimensional representation~$\Psi$ of~$\G$, obtaining trace 
formulae for vector valued modular forms, but for simplicity we restrict 
ourselves to classical modular forms. Since we only use the structure 
of~$\PSL_2(\Z)$ as a free group with two elliptic generators and having one cusp, 
the same method easily applies to prove trace formulas for other 
Hecke groups.  

There is a vast literature on the trace formula for Hecke operators for 
congruence subgroups, and previous authors compute the trace on 
the cuspidal subspace alone. One insight of the present paper which is apparent
in the dimension formula above, and was first observed by
Zagier~\cite{Z1,Z} in the case of $\SL_2(\Z)$, is that 
the trace formula is much simpler for the linear 
combination~\eqref{0}. From our point of view this is reflected in the 
fact that formula~\eqref{1} involves only scalar, elliptic and split 
hyperbolic conjugacy classes, those for which $\e_\G(X)\ne 0$. 
One can also extract the cuspidal contribution from~\eqref{0} by computing 
the trace on the Eisenstein part, but since it is rather technical and 
would double the size of this paper, we leave this computation to the 
sequel~\cite{P}.
As a consequence of the main result of this paper, there we obtain 
explicit formulas in terms of class numbers for the trace of a composition of 
Hecke and Atkin-Lehner operators for~$\G_0(N)$. The formulas
obtained are among the simplest in the literature, and unlike in previous 
work, we need no restrictions on the index of the operators involved. We 
refer to~\cite{P} for a bibliography of previous results on the trace formula.

In the remainder of the introduction, we state the main theorem and give two 
applications. 

\subsection{Statement of results}
We start with some definition and notations in use throughout the paper.

The action of the double coset operator $[\SS]$ on $M_k(\G,\chi)$ 
is defined using a multiplicative function $\wc$ on the semigroup generated 
by~$\G$ and~$\SS$ inside the commensurator~$\wg$, such that $\wc|_{\G}=\chi^{-1}$, namely
 \be\label{chi} \wc(\g\ss\g')=\chi^{-1}(\g\g')\wc(\ss)\;, \quad \text{ for all } \g\in\G, \ss\in\SS \;. \ee  
A modular form $f\in M_k(\G,\chi)$ satisfies $f|_k \g=\chi(\g) f$, and the double coset $\SS$ 
defines an operator~$[\SS]$ on $M_k(\G,\chi)$ by
  \be \label{hecke}
  f|[\SS]=\sum_{\ss\in\G\backslash\SS} \det\ss^{k-1} \cdot \wc(\ss)\cdot f|_k \ss \;, \ee
where $f|_k\g (z)=f(\g z) (c_\g z+d_\g)^{-k}$, and we write 
$\g=\sm {a_\g}{b_\g}{c_\g}{d_\g}$ throughout the paper. In \eqref{hecke} and 
in similar sums over right cosets, it is understood (and usually obvious)
that the sum is independent of the coset representatives chosen. 
\begin{example*} Let $\G$ be the congruence subgroup $\G_0(N):=\{\g\in\G_1 : N|c_\g\}$, and
let $\chi$ be a character modulo $N$ viewed as a character of 
$\G_0(N)$ by $\chi(\g)=\chi(d_\g)$. The usual Hecke operators $T_n$ on $M_k(\G,\chi)$ 
are associated to the double coset
  \be \label{delta}  \DD_n:=\{\ss\in M_2(\Z) \;:\; \det \ss=n,\ N|c_\ss,\ (a_\ss,N)=1\}\;, \ee
and $\wc(\ss)=\chi(a_\ss)$ for $\ss\in\DD_n$.   
\end{example*}
For any subset $\Sc$ of $\GL_2^+(\R)$, we denote by
$\ov{\Sc}=(\Sc\cup -\Sc) /\{\pm 1\}\subset \GL_2^+(\R)/\{\pm 1\}$ its 
projectivization. 

For a $\ov{\G}$-conjugacy class $\X\subset  \GL_2^+(\R)/\{\pm 1\}$, we let 
$M_\X\in \GL_2^+(\R) $ be the lift of any representative.
We denote by $\DD(\X)=\tr(M_\X)^2-4\det(M_\X)$ the 
discriminant of the quadratic form associated to $M_\X$, and by 
$|\stab_{\ov{\G}}M_\X|$ the (possibly infinite) cardinality of the stabilizer 
of~$M_\X$ under conjugation by~$\ov{\G}$. We introduce the conjugacy class 
invariant 
\be\label{1.50}
\e_\G(\X)=\begin{cases} \phantom{xx}
\dfrac{|\G\backslash\H| }{2 \pi} &  \text{if $M_X$ scalar,}\vspace{2mm} \\
\dfrac{\sgn \DD(\X)}{|\stab_{\ov{\G}} M_\X|} &  \text{ otherwise,}           
            \end{cases}
\ee
where $|\G\backslash\H|$ is the area of a fundamental domain for $\G$ with respect 
to the standard hyperbolic metric, and we use 
the convention that $1/\infty=0$. Any double coset $\SS\subset \wg$ 
contains only finitely many conjugacy classes~$\X$ with $\e_\G(\X)\ne 0$, 
namely the elliptic, scalar, and split hyperbolic classes (namely those that 
contain an element fixing two distinct cusps of $\G$, for which $\e_\G(X)=1$).  

Let $p_w(t,n)$  be the Gegenbauer polynomial defined by the power series 
expansion 
  \be \label{gegen}(1-tx+nx^2)^{-1}= \sum_{w\ge 0} p_w(t,n)x^w\;. \ee 
  
We can now state the main theorem of this paper. 

\begin{thm} \label{TF}
Let $\G$ be a finite index subgroup of $\G_1$, $k\ge 2$ an integer, 
$\chi$ a character of~$\G$ with kernel of finite index in~$\G$, 
and~$\SS$ a double coset of $\G$ such that $|\G\backslash\SS|=|\G_1\backslash\G_1\SS|$. 
Assuming $\chi(-1)=(-1)^k$ if $-1\in\G$, we have 
 \be \label{TFF} \begin{split}
 \tr([\SS], M_k(\G,\chi)+ S_k^c(\G,\chi)) \,=\,
 \sum_{\X\subset \ov{\SS}}p_{k-2}(\tr M_\X, \det M_\X)\;\wc(M_\X)\;\e_\G(\X)   \\
 \+ \delta_{k,2}\delta_{\chi,{\bf 1}}\; \sum_{\ss\in\G\backslash\SS} \wc(\ss)\;,
 \end{split}
 \ee
where the sum is over $\ov{\G}$-conjugacy classes $\X$ in $\ov{\SS}$ with representative $ M_\X\in \SS$. 
The symbol $\dd_{a,b}$ is 1 if $a=b$ and 0 otherwise. 
\end{thm}
\noindent If $\G$, $\SS$ and $\chi$ are invariant under conjugation by an
order 2 element of determinant~$-1$, then we can replace the space of 
anti-holomorphic cusp forms~$S_k^c(\G,\chi)$
by $S_k(\G,\chi)$ in the theorem, as well as in all the trace formulas in 
the paper (see Remark~\ref{r3.1}).

We state the theorem under the assumptions used in the proof, but 
we expect the same formula to hold for any Fuchsian subgroup of 
the first kind~$\G$, and any double coset $\SS\subset \wg$, with the 
term on the second line multiplied by 2 if~$\G$ has no cusps. 

What we prove is an equivalent version of Theorem \ref{TF} in 
which the sum is over $\G_1$-conjugacy classes $X$, and we set
 \be\label{1.70} \e(\X):=\e_{\G_1}(\X)\;.\ee
Explicitly, if $M_X\in X$ is any representative, then $\e(X)$ is equal to: 
1/6 if $M_X$ is scalar;
$-1/|\stab_{\ov{\G}_1} M_X|$ if $M_X$ is elliptic; 1 if $M_X$ is hyperbolic 
fixing two cusps of~$\G_1$; and 0 otherwise. 
\setcounter{thm}{0}
\begin{thm}[Second version]\label{T1}
Under the assumptions in the first version, we have 
 \be \label{TF2} 
 \begin{split}
 \tr([\SS], M_k(\G,\chi)+ S_k^c(\G,\chi)) \,=\,
 \sum_{\X}p_{k-2}(\tr M_\X, \det M_\X)\;\cc_{\G,\SS}^\chi(M_\X)\; \e(\X) \\
 \+\delta_{k,2}\delta_{\chi,{\bf 1}}\; \sum_{\ss\in\G\backslash\SS} \wc(\ss)\;,
 \end{split}
 \ee
where the sum is over $\ov{\G}_1$-conjugacy classes $\X\subset \ov{\G_1\SS\G_1}$ with representative $M_\X\in \G_1\SS\G_1$, 
and\footnote{The same sign is chosen in all three places in~\eqref{1.10}. 
If $-1\notin\G$ at most one choice of signs is possible for each $A$, while 
if $-1\in \G$ both choices yield the same value for the summand.}
 \be\label{1.10} \cc_{\G,\SS}^\chi(M)\; :=\; \sum_{\substack{A\in \ov{\G}\backslash\ov{\G}_1\\ 
 \pm AMA^{-1}\in \SS}}(\pm 1)^k\wc(\pm AMA^{-1}) \;.\ee
\end{thm}
\noindent Theorem~\ref{T1}, as well as the equivalence of the two versions, 
is proved in Section~\ref{s2.1}. The coefficient 
$\cc_{\G,\SS}^\chi(M)=(-1)^k \cc_{\G,\SS}^\chi(-M)$ also depends on
the parity of $k$, but for simplicity we suppress this dependence from the notation. 

Next we give two applications of the trace formula, in which we vary the weight 
or the level. When varying the weight $k$, we obtain that the generating series of the traces of 
Hecke operators is a rational function, as observed in special cases
in~\cite{Z1,FOP}.  
\begin{thm}\label{T2}
Set ${\bf T}_{\G,\SS}^\chi(k)=\tr([\SS], M_k(\G,\chi)+S_k^c(\G,\chi) )$. 
When $-1\notin\G$, the formula in Theorem~\ref{TF} 
is equivalent to the following power series identity
\[
 \sum_{k\ge 2} {\bf T}_{\G,\SS}^\chi(k)\cdot x^{k-2}=
 \sum_{\X\subset \ov{\SS}}\frac{\wc(M_\X)\; \e_\G(\X)}{1-\tr(M_\X)x+\det(M_\X) x^2} 
  +\delta_{\chi,{\bf 1}}\sum_{\ss\in\G\backslash\SS} \wc(\ss)\;,
  \] 
where the sum is over $\G$-conjugacy class 
$\X\subset \ov{\SS}$ with representatives $M_\X\in\SS$. When $-1\in\G$ a similar 
formula holds, with the first summand in the right side replaced by 
\[\frac{\e_\G(\X)\wc(M_\X)}{2} \Big(\frac{1}{1-\tr(M_\X)x+\det(M_\X) x^2} 
+ \frac{\chi(-1) }{1+\tr(M_\X)x+\det(M_\X) x^2} \Big) .
\]
\end{thm}

When varying the level, the simplest formula is obtained when $\G$ varies 
through the principal congruence subgroups 
$\G_n=\{\g\in \G_1: \g\equiv I \pmod{n}\}$. The next theorem seems to be 
the first explicit trace formula for~$\G_n$. 
\begin{thm}\label{T3} Fix $\ss=\sm \ell 001$ with $\ell>1$, and let $n>2\ell+2$ 
coprime to $\ell$. Setting $n'=n/\gcd(n,\ell-1) $, 
we have  
\[\tr ([\G_n\ss\G_n], M_k(\G_n)+S_k(\G_n) )=\frac{(\ell^{k-1}-1)}{2}\cdot
\frac{\varphi(n')\varphi_2(n)}{\varphi_2(n')}+\delta_{k,2}\varphi_1(\ell),\]
where $\varphi_1(\ell)=[\G_1:\G_0(\ell)]$, $\varphi_2(n)=[\ov{\G}_1:\ov{\G}_n]$
and $\varphi$ is Euler's phi function. 
\end{thm}
In particular, for $\ell$ prime the operator in the theorem is the usual Hecke 
operator~$T_\ell$ and we obtain the limit formula~\eqref{3}. The proof of Theorem~\ref{T3} is given 
in Section~\ref{sec5}.

\noindent {\bf Acknowledgements.}
I am grateful to Don Zagier for introducing me to the period polynomial approach 
for proving the Eichler-Selberg trace formula for $\SL_2(\Z)$, and for sharing generously his insights. 
I also thank Vicen\c{t}iu Pa\c{s}ol for many illuminating 
discussions on the subject of period polynomials and modular symbols. 

This work was partly supported 
by the European Community grant PIRG05-GA-2009-248569 and by 
the CNCS grant PN-II-RU-TE-2011-3-0259. Part of this work was completed 
during several visits at MPIM in Bonn, whose support I gratefully acknowledge.

\section{A trace formula on the parabolic cohomology of the modular group}
\label{s2}

Let $\V$ be a $\G$-module, with $\G$ denoting $\SL_2(\Z)$ in this section only. 
We give a formula for the trace of Hecke operators on the parabolic cohomology 
group $H_P^1(\G, \V)$, by relating it to the trace of a certain operator 
$\wT_n$ on the \emph{period subspace} $\W$ of $\V$ (Corollary~\ref{c2.1}). 
The trace on $\W$ is computed in Theorem~\ref{Tcoh}, using a special 
operator $\wT_n$, studied in detail in~\cite{PZ}.

\subsection{Double coset operators on cohomology}
To define the action of double coset operators on cohomology, let us consider 
more generally a group $\G$ and a 
right $\G$-module $\V$, which is assumed to be a vector space over $\C$. Let $\SS$ be a double coset of 
$\G$ contained in the commensurator of $\G$ inside 
a larger ambient group, so that the number of right cosets $|\G\backslash \SS|$ is
finite. Assume that elements in $\SS$ act on $\V$ in a way compatible with the action of $\G$, that is
  \be \label{comp} P|(gM)=(P|g)|M,\quad P|(Mg)=(P|M)|g ,\text{ for } P\in \V, g\in\G, M\in \SS \;,\ee
namely $\V$ is a module for the semigroup generated by $\G$ and $\SS$ inside the 
commensurator of $\G$. Fix representatives $M_K\in \SS$ for cosets $K\in \G\backslash \SS$, 
and for $\g\in\G$, let $\g_K\in\G$ be the unique element such that $M_K\g^{-1}=\g_K^{-1} M_{K\g^{-1}}$. 
If $\phi:\G\rightarrow \V$  is a cocycle, namely $\phi(gh)=\phi(g)|h+\phi(h)$, we define 
  \be\label{131} \phi|[\SS](\g)= \sum_{K\in\G\backslash \SS} \phi(\g_K)| M_K \;.\ee
Then $\phi|[\SS]$ is a cocycle, whose cohomology class is independent of the choice of 
representatives~$M_K$. If $\G$ is a subgroup of $\SL_2(\R)$ and $\phi$ is a 
parabolic cocycle (that is $\phi(\g)=v_\g|1-\g$ for all parabolic elements $\g$), 
so is $\phi|[\SS]$, which defines an action of $[\SS]$ on the parabolic 
cohomology group~$H_P^1(\G, \V)$~\cite{Hab,Sh}.

\begin{remark}\label{r1}
If $\G$ contains an element $J$ in the center with $J^2=1$ (e.g., $J=-1$ for $\G$ a subgroup of $\SL_2(\R)$), let 
$\V^{J}$ be the subspace of $J$-invariants in $\V$. For any
cocycle $\phi:\G\rightarrow \V$ and $\g\in G$ we have $\phi(\g)|1-J= \phi(J)|1-\g$, so $\g\mapsto \phi(\g)|1-J$
is a coboundary. Therefore $\phi'(\g)=\frac 12 \phi(\g)|(1+J)$ is a cocycle in the same class as $\phi$, which takes
values in the module $\V^J$, and the map 
$$H^1(\G, \V)\simeq H^1(\G, \V^J), \quad [\phi]\mapsto [\phi']\;,$$
is an isomorphism. The same is true for parabolic cohomology, if $\G$ is a subgroup of $\SL_2(\R)$.
Therefore we can restrict without loss of generality to modules on which $J$ acts identically, 
\end{remark}

\subsection{Hecke operators on the period subspace}\label{s2.2}
In this section we set $\G=\SL_2(\Z)$. Let $\V$ be a right $\G$-module 
on which~$-1$ acts identically, and let $\SS$ be a double coset 
of~$\G$ whose elements act on~$\V$ by an action denoted $|$, 
in a way compatible with the action of $\G$ as in \eqref{comp}. 
By linearity we also have an action of elements of $\RR_\SS=\Q[\ov{\SS}]$ on $\V$. 

Let $S=\sm 0{-1}10$ and $T=\sm 1101$ be two generators of $\G$, and let 
$U=TS$, an element of order~3 in $\ov{\G}=\PSL_2(\Z)$. To ease notation,
we use the same notation for an element in $\ov{\G}$ and for a lift of it
in $ \G$. Any parabolic cocycle 
$\varphi: \G\rightarrow \V$ can be modified by a coboundary so that $\varphi(T)=0$, and 
then the element $P=\varphi(S)=\varphi(TS)$ belongs to the subspace
  $$\W\;:=\;\{P\in\V\,:\, P|(1+S)=P|(1+U+U^2)=0 \}\,,$$
called \emph{the period subspace}. Conversely, if $P\in\W$, the map 
$\varphi_P:\G\rightarrow \V$ with $\varphi_P(T)=0$, $\varphi_P(S)=P$ extends to a 
parabolic cocycle via $\varphi_P(gh)=\varphi_P(g)|h+\varphi_P(h)$. This gives an exact sequence 
\be \label{es}
0\longrightarrow \cc \longrightarrow \W \xrightarrow{\;P\mapsto [\varphi_P]\;} H_P^1(\G, \V) \longrightarrow 0\;,
\ee
where $\cc=\{P|(1-S)\;:\; P \in \V,\ P|(1-T)=0 \}\subset \W$ is called the \emph{coboundary subspace}.   

The action of the double coset operator $[\SS]$ on cohomology was 
defined in the previous section, and now we show directly that the corresponding 
action on~$\W$ is determined by the same elements introduced by Choie and Zagier to
express the action of Hecke operators on period polynomials of modular forms. 
In each coset  $K\in \G\backslash \SS$ (identified as above with $\ov{\G}\backslash\ov{\SS}$),  
we choose a representative~$M_K$ that fixes the cusp infinity, and we let 
\be\label{2.4}
T_{\SS}^\infty=\sum_{K\in \G\backslash\SS} M_K \in \RR_{\SS} \,.
\ee
Then there exist elements $\wT_{\SS}\in \RR_{\SS}$ satisfying 
\begin{equation}\label{Z1}\tag{A}
(1-S)\wT_{\SS}-T_{\SS}^\infty(1-S) \in (1-T)\RR_{\SS}\;.
\end{equation}
This was shown in ~\cite{CZ} in the case $\SS=\M_n$, and the proof for general
$\SS$ is similar. 
\begin{prop}\label{p2.1} 
Any element $\wT_{\SS}\in \RR_{\SS}$  satisfying property~\eqref{Z1}
preserves the space~$\W$, and its action on $\W$ corresponds to the action of $[\SS]$ on $H_P^1(\G, \V)$
via the map in \eqref{es}, namely the cocycles $\varphi_P|[\SS]$ and $\varphi_{P|\wT_\SS}$ are in the same 
cohomology class, for any $P\in \W$. 
\end{prop}
We use the following easy lemma, which we state in greater generality than needed here. 
\begin{lemma*} Let $G$ be a group with generators $g_1, \ldots, g_r$, let $V$ be a right $G$-module viewed
also as a $\Q[G]$-module, and let $\varphi:G\rightarrow V$ be a cocycle. Then for each $g\in G$ there 
exist $X_i\in \Q[G]$ such that in the group algebra $\Q[G]$ we have $1-g=\sum_{i=1}^r (1-g_i) X_i$, 
and for any such $X_i$ we have 
  $$\varphi(g)=\sum_{i=1}^r\varphi(g_i)|X_i\;.$$
\end{lemma*}
\begin{proof}Both the existence of $X_i$ as above and the relation follow 
by induction on the length of~$g$ as a product in the generators $g_i$: 
we have $1-g_ig=1-g +(1-g_i)g$, and $\varphi(g_i g)=\varphi(g_i)|g+\varphi(g)$.
\end{proof}
\begin{proof}[Proof of Proposition \ref{p2.1}.]
We use the representatives $M_K$ in~\eqref{2.4} to define the action of $[\SS]$ on $\varphi$
in \eqref{131}. 

Let $\varphi\in Z_P^1(\G,\V)$ with $\varphi(T)=0$. Writing $M_K T^{-1}=T_{K}^{-1} M_{KT^{-1}}$, 
we have $T_K \infty=\infty$, so $\varphi(T_K)=0$, and \eqref{131} shows that $\varphi|[\SS](T)=0$.
Therefore it remains to show that $\varphi|[\SS](S)=\varphi(S)| \wT_\SS \;,$ if $\wT_\SS$ 
satisfies~\eqref{Z1}, which would show in particular that $\wT_\SS$ preserves $\W$. 

Let $\wT_\SS=\sum_{K\in \G\backslash \SS} X_K M_K $ satisfying \eqref{Z1}, with $X_K\in\Q[\G]$. 
Relation \eqref{Z1} implies
  \[(1-S)X_K M_K \equiv M_K-M_{KS^{-1}}S= (1-S_K)M_K   \pmod{(1-T)\RR_n} \;,\]
where $M_K S^{-1}= S_K^{-1} M_{KS^{-1}}$. We have therefore $1-S_K= (1-S)X_K +(1-T)Y_K$, and the lemma applied to the
group $\ov{\G}$ with generators $S$, $T$ and to the cocycle $\varphi$ with $\varphi(T)=0$ gives  
  \[\varphi(S)| \wT_\SS=\sum_{ K\in \G\backslash\SS}(\varphi(S)| X_K)| M_K
  =\sum_{K\in\G\backslash\SS} \varphi(S_K)| M_K = \varphi|[\SS](S) \;,  \]
where we used \eqref{131}.
\end{proof}
\begin{cor}\label{c2.1}
For any element $\wT_\SS\in \RR_\SS$ satisfying~\eqref{Z1} we have 
\be\label{2.5}\tr([\SS], H_P^1(\G, \V)) + \tr(\wT_\SS,\cc) =\tr (\wT_\SS, \W). \ee   
\end{cor}
\begin{proof}
This is immediate from~\eqref{es}, once we show that the coboundary subspace 
  $\cc$ is also preserved by $\wT_\SS$. Indeed, if $P|1-T=0$, by  \eqref{Z1} we have 
  \be \label{2.3} P|(1-S)|\wT_\SS=P| T_{\SS}^\infty |(1-S)\, \ee
and since $T_\SS^\infty(1-T)\in (1-T)\RR_n$ we also have $P| T_\SS^\infty|1-T=0$.
\end{proof}
\begin{remark}\label{r3.2}
The previous proof shows that the following exact sequence is Hecke-equivariant
\[
0\longrightarrow \V^{\G} \longrightarrow \D \xrightarrow{\;P\mapsto P|1-S\;} \cc \longrightarrow 0\;,
\]
where  $\D:=\{P\in \V\;:\; P|1-T=0 \}\simeq H^0(\G_{\infty}, \V)$, and the 
Hecke action is by $\wT_\SS$ on $\cc$ and by $T_{\SS}^\infty$ on the first two terms. 
Therefore we have $\tr(\wT_\SS,\cc)= \tr(T_{\SS}^\infty,\D)-\tr(T_{\SS}^\infty,\V^{\G} )$,
and one can compute explicitly the right hand side for the modules of interest. 
\end{remark}

By Corollary~\ref{c2.1}, computing the trace of $[\SS]$ on the parabolic cohomology reduces  
to computing the trace of $\wT_\SS$ on the period subspace. The latter trace is computed using an operator~$\wT_\SS$ 
satisfying an extra property introduced in~\cite{Z}:
\begin{equation}\label{Z23}\tag{B} 
\begin{cases}   \quad\qquad\wT_\SS(1+S) &\in \quad(1+U+U^2) \RR_\SS\, , \\  
                 \quad \wT_\SS (1+U+U^2) &\in\quad (1+S)\RR_\SS\, .
\end{cases}
\end{equation}
It is easy to show that there exist operators satisfying both~\eqref{Z1} 
and~\eqref{Z23}, and the main difficulty in this approach to the trace 
formula is contained in the next theorem, proved in the upcoming paper~\cite{PZ}.
Recall that $\G=\SL_2(\Z)$ in this section. 
\begin{theorem}[\cite{PZ}] \label{T2.1}
Let $\wT_\SS=\sum_{M} c(M) M\in\RR_\SS$ be any element 
satisfying~\eqref{Z1} and~\eqref{Z23}.  

\emph{(a)} For each \emph{right coset} $K\in \ov{\SS}/\ov{\G}$ we have $\sum_{M\in K} c(M)=-1$.  
   
\emph{(b)} For each conjugacy class $\X\subset \ov{\SS}$ we have 
    \be\label{Z4}\tag{C} \sum_{M\in \X} c(M)=\e(\X)\;, \ee
where $\e(\X)$ is defined in~\eqref{1.70}. 
\end{theorem}
Part (a) is easy 
to prove, and the main difficuly is to show that any element satisfying~\eqref{Z1} and~\eqref{Z23} 
also satisfies~\eqref{Z4}. Note that it is enough to produce such an operator 
for the double coset $\M_n$ of integral matrices of determinant $n$, because 
any primitive double coset $\SS=\G\ss\G$ can be scalled such that 
$\SS\subset \M_n$ for some $n$, and then the relations~$\eqref{Z1}-\eqref{Z4}$ 
for the operator $\wT_n$ associated to the double coset~$\M_n$ imply the corresponding relations 
for~$\wT_\SS$.

Before stating the main theorem of this section, we define 
$T_\SS\in \RR_\SS$ to be the sum
of a complete system of coset representatives for $\G\bsh\SS$. The operator 
$T_\SS$ acts (on the right) on the invariant space~$\V^\G$, and the action is clearly independent 
of the coset representatives chosen. 

\begin{theorem}\label{Tcoh} Let $\V$ be a $\G$-module with period subspace
$\W$, and let~$\SS$ be a double coset acting on $\V$ as in~\eqref{comp}.  
Assume that $\V$ admits a nondegenerate, $\G$-invariant pairing.
If $\wT_\SS\in \RR_\SS$ is any element satisfying~\eqref{Z1}, we have
  \[  \tr (\wT_\SS, \W)=\tr(T_{\SS} , \V^{\G}  ) 
  + \sum_{\X\subset\ov{\SS}} \tr( M_\X, \V)\; \e(\X)\;,\]  
where the sum is over $\ov{\G}$-conjugacy classes $\X$ in $\ov{\SS}$ 
with representatives  $M_\X\in \SS$.
\end{theorem}

\begin{proof} 
We use Theorem~\ref{T2.1}. The trace on $\W$ is the same for any element satisfying~\eqref{Z1}, 
and we choose $\wT_\SS$ satisfying~\eqref{Z23} and~\eqref{Z4} as well. 
Property~\eqref{Z23} implies that~$\wT_\SS$ maps the spaces $A=\ker(1+S)$ and 
$B=\ker(1+U+U^2)$ into each other, and basic linear algebra shows that 
\be \label{1.9}\tr(\wT_\SS,A\cap B)=\tr(\wT_\SS,A+B) .\ee 
We have $\W=A\cap B$, and denote $\V'=A+B$. 

From the $\G$-invariance of the pairing, we have $\ker(1-S) \subseteq A^\perp$. 
Since  $\ker(1-S)=\im(1+S)$, the nondegeneracy of the pairing implies that 
$\ker(1-S)$ and $A^\perp$ have the same dimension, hence they are equal. 
Similary $B^\perp= \ker(1-U)$, so $(A+B)^\perp=A^\perp\cap B^\perp= \V^{\G}$ 
as~$\G$ is generated by $S$~and~$U$. Therefore we have a direct sum decomposition
\be \label{2.8} \V=\V'\oplus \V^\G \,. \ee

Write $\wT_\SS=\sum_{C\in \SS/\G } R_C X_C  $, where $R_C\in \SS$ is any 
representative for the coset $C$, and $X_C\in \Q[\G]$. 
We chose the representatives $\{R_C\}$ so that they also form
a system of representatives for the left cosets $\G\backslash\SS$ \cite[Lemma 3.5]{Sh},
so that we can choose $T_\SS=\sum_C R_C$ in the statement.   

The element $\wT_\SS$ does not preserve $\V^\G$, so we decompose for $P\in \V^\G$ 
and a coset $C\in \SS/\G$:
$$P|R_C=P_C + P'_C, \text{ with $P_C\in \V^\G$, $P'_C\in \V'$.}$$ 
As $P_C\in \V^\G$, it follows by Theorem~\ref{T2.1} (a) that $P_C|X_C=-P_C$, and 
we obtain $P|R_C X_C=-P_C + P''_C$ with $P''_C=P'_C|X_C \in \V'$ 
(as $\V'=\im(1-S)+\im(1-U)$ is invariant under the action of $\Q[\G]$). Therefore
\[P|\wT_\SS= -\sum_C P_C+\sum_C P''_C, \quad P|T_\SS=\sum_C P_C+\sum_C P'_C\;.
\]
Since $T_\SS$ preserves $\V_\G$, we conclude from the direct sum 
decomposition~\eqref{2.8} that $P|T_\SS=\sum_C P_C$, 
and therefore~\eqref{1.9} and the previous relation give
\[ \tr(\wT_\SS,\W)=\tr(\wT_\SS, \V')=\tr(T_\SS,\V^\G)+\tr(\wT_\SS,\V) \,. \]
By Theorem~\ref{T2.1} (b), the last term can be written as a sum over conjugacy classes as 
in the statement of the theorem, finishing the proof. 
\end{proof}

\section{A trace formula on the space of period polynomials}\label{s2.0}  

Let $\G_1=\SL_2(\Z)$. We now specialize the $\G_1$-module of the last section
to be the induced module 
$\ind_\G^{\G_1}(\sym^{w} \C^2\otimes \chi)$, where $\G\subset \G_1$ is a 
finite index subgroup (note the change in notation from the last section). Its period subspace is the space of 
period polynomials associated with the space of cusp forms $S_{w+2}(\G,\chi)$, 
and we apply the Eichler-Shimura 
isomorphism and the Shapiro lemma to show that the trace of Hecke operators on 
the period subspace equals the trace on $M_{w+2}(\G,\chi)+S_{w+2}^c(\G,\chi)$ 
(Proposition~\ref{L2.1}). We then apply Theorem~\ref{Tcoh} to compute this trace 
in Section~\ref{s2.1}, thus proving Theorem~\ref{T1}. 

For other uses of the Eicher-Shimura isomorphism together with the Shapiro lemma 
in the study of modular forms for congruence subgroups see~\cite{Ha, He}.

\subsection{The Eichler-Shimura isomorphism}\label{s3.1}
Let $\G$ be a finite index subgroup of $\G_1$, and $\chi$ a character of $\G$ whose kernel has finite index in $\G$. 
For $w=k-2\ge 0$, we view the space $V_w$ of complex polynomials of degree $\le w$ as a $\G$-module by 
 \be\label{2.1} P|_\chi\g:= \chi(\g^{-1}) P|_{-w} \g, \text{ for } P\in V_w, \g\in\G\;,\ee
and we denote it by $V_w^\chi$ to indicate this action. 

Let $\SS\subset \M$ be a double coset such that $\SS=\SS\G \SS$, and $\SS$ is a disjoint finite union of right cosets. 
Using a function $\wc$ as in \eqref{chi}, we define 
an operation of elements $M\in \SS$ on $V_w^\chi$ by $$P|_{\chi}M=\wc(M)P|_{-w} M,$$ and this operation 
is obviously compatible as in~\eqref{comp} with the action of $\G$. Therefore we have an operator~$[\SS]$ 
acting on $H_P^1(\G, V_w^\chi)$ as in~\eqref{131}.

In order to state the Eichler-Shimura isomorphism, let $S_k^c(\G,\chi)$ be the space of anti-holomorphic cusp forms 
$\ov{S_k(\G,\ov{\chi})}$, where the bar denotes complex conjugation. Functions $g\in S_k^c(\G,\chi)$ are anti-holomorphic
on the upper half plane, and satisfy $g|^c_k \g=\chi(\g)g$ for $\g\in\G$, where 
\[ g|^c_k \g(z):=g(\g z)j(\g,\ov{z})^{-k}.\]
The operator $[\SS]$ acts on $S_k^c(\G,\chi)$ as in~\eqref{hecke},
with the action $|_k$ replaced by $|_k^c$. 

\begin{theorem} [Eichler-Shimura] We have a Hecke-equivariant isomorphism
\[S_k(\G,\chi)\oplus S_k^c(\G,\chi)\longrightarrow H_P^1(\G,V_w^\chi) 
\]
given by $(f,g)\mapsto [\vp_f]+[\vp_g^c]$ with the parabolic cocycles $\vp_f$, $\vp_g^c$ defined by 
  \[\vp_f(\g)(X)=\int_{\g^{-1}z_0}^{z_0} f(t)(t-X)^w dt,\quad 
  \vp_g^c(\g)(X)=\int_{\g^{-1}z_0}^{z_0} g(t)(\ov{t}-X)^w \ov{dt}\;,\]
where $z_0\in \H\cup \{\mathrm{cusps}\}$. 
\end{theorem}
\begin{proof} The proof is given by Shimura in \cite[Sec.~8]{Sh}, with the difference that here
$\G$ acts on the right on $V_w^\chi$ instead of on the left in loc. cit. Shimura defines the action of the 
Hecke operator $[\SS]$ using a character $\chi$ of the semigroup generated  by $\G$ and $\SS^\vee$ inside $\wg$, 
where $\SS^\vee$ is the adjoint of~$\SS$. Our function~$\wc$ is related to such a $\chi$ by $\wc(\ss)=\chi(\ss^\vee)$ for $\ss\in\SS$.
\end{proof}

\begin{remark}\label{r3.1}
Assume that there exists $\eta\in \GL_2(\R)$ with $\eta^2=1$ and $\det \eta=-1$, 
such that 
  \[ \begin{cases}\qquad
    \eta\G\eta=\G,\quad\eta\SS\eta =\SS \\ 
    \chi(\eta\g\eta)=\chi(\g),\quad \wc(\eta\ss\eta)=\wc(\ss)\quad \text{ for }\g\in\G,\ \ss\in\SS. 
    \end{cases} \]
For example, if $\G=\G_0(N)$ we can take $\eta=\sm {-1}001$. Under this assumption, 
we have a Hecke-equivariant isomorphism 
  \[ S_k(\G,\chi)\longrightarrow S_k^c(\G,\chi),\quad f\mapsto f^*(z)=f(\eta \ov{z}) j(\eta,\ov{z})^{-k}\;,
  \] 
with inverse $g\mapsto g^*$, $g^*(z)=g(\eta \ov{z}) j(\eta,z)^{-k}$, so in this case
all the trace formulas in this paper hold with the space $S_k^c(\G,\chi)$ 
replaced by $S_k(\G,\chi)$.
\end{remark}

\subsection{The Shapiro isomorphism} 
The Shapiro lemma gives a Hecke equivariant isomorphism between the cohomology
groups $H_P^1(\G_1, \ind_{\G}^{\G_1} V_w^\chi)$ and $H_P^1(\G,V_w^\chi)$, 
and we now describe it explicitly. 

Since~$V_w$ is a $\G_1$-module as well as a $\G$-module, we will identify 
the induced module $\ind_{\G}^{\G_1} V_w^\chi$ with the space of functions 
$P: \G_1\rightarrow V_w^\chi$ such that 
  \[ P(\g A)=\chi(\g) P(A)\;,\ \ A\in \G_1, \g\in\G \;, \]
on which~$\G_1$ acts by $P|g(A)= P(Ag^{-1})|_{-w} g$. By Remark~\ref{r1}, 
the cohomology group $H_P^1(\G_1, \wV_{w}^{\G,\chi})$ does not change upon 
replacing this module with its subspace $V_w^{\G, \chi}$ on which 
$-1$ acts trivially: 
  $$V_w^{\G, \chi}:=\left\{P:\G_1\rightarrow V_w \left|
  \begin{matrix}P(-A)=(-1)^w P(A)\\
     P(\g A)=\chi(\g)P(A)
  \end{matrix} \ , \text{ for } \g\in \G, A\in \G_1\right.\right\}.$$  
  
We make the following assumption on the double coset~$\SS$.

\begin{hypo}\label{eq_star} The map 
$$\G\backslash\SS\longrightarrow \G_1\backslash\G_1\SS,\quad   \G \sigma \mapsto
 \G_1\sigma  $$  
is bijective,  or equivalently $|\G\backslash\SS|= |\G_1\backslash\G_1\SS|$.
\end{hypo}
This assumption is satisfied by the double cosets giving the usual Hecke 
and Atkin-Lehner operators for the congruence subgroups $\G_1(N)$ and 
$\G_0(N)$. It is related to the notion of compatible Hecke pairs in~\cite{AS}. 

Under Assumption \ref{eq_star}, we define an action of elements $M\in \M$ on 
$V_w^{\G,\chi}$, which for $\chi=\bf{1}$ is the same as in \cite[Sec.~5]{PP}.
For $A\in\G_1$ with $M A^{-1}\in \G_1 \SS$, let $A_M\in\G_1$, $M_A\in\SS$ such that 
$M A^{-1} =A_M^{-1} M_A$ and define:
\begin{equation}\label{eq_act1}
P|_{\SS}M(A)=\begin{cases} \wc(M_A) P(A_M)|_{-w}M  
 & \text{ if }  M A^{-1}\in \G_1 \SS\\
        0 & \text{ if } M A^{-1}\notin \G_1 \SS\;.
         \end{cases}
\end{equation}
By Assumption \ref{eq_star} and \eqref{chi}, the definition does not 
depend on the decomposition $M A^{-1} =A_M^{-1} M_A$. This is not a 
proper action of the semigroup $\M$, but it is compatible with the 
action of~$\G_1$ as in~\eqref{comp}. 

Formula~\eqref{eq_act1} defines an action of the $\G_1$-double coset~$\G_1\SS\G_1$ 
on $V_w^{\G,\chi}$, which is compatible with the action of $\G_1$ as in~\eqref{comp}. 

\begin{theorem}[Shapiro Isomorphism] For $\SS$ a double coset 
satisfying Assumption~\ref{eq_star}, we have a Hecke-equivariant isomorphism
 \be \label{shap}  H_P^1(\G_1, V_w^{\G,\chi})\simeq H_P^1(\G,V_w^\chi)\;, \ee
with $[\SS]$ acting on the right side, and $[\G_1\SS \G_1]$ acting on 
the left side by~\eqref{eq_act1}.
\end{theorem}
\begin{proof} The isomorphism is given on cocycles by
$[\vp]\mapsto [\vp']$, where $\vp'(\g)=\vp(\g)(1)$, and the Hecke-equivariance
is easily verified. See also \cite[Lemma 1.1.4]{AS}.
\end{proof}

\subsection{Trace on period polynomials}

For $\V=\V_w^{\G,\chi}$, we denote by $W_w^{\G, \chi}$ the period 
subspace defined in Section~\ref{s2}. Let $\SS_1:=\G_1 \SS \G_1$,
and assume that $\SS\subset \M_n$. An operator $\wT_{\SS_1}$ satisfying~\eqref{Z1} 
acts on $W_w^{\G, \chi}$ via~\eqref{eq_act1}, and its action is the same as that 
of the ``universal operator'' $\wT_n$ 
that satisfies~\eqref{Z1} for the double coset~$\M_n$, 
since matrices in $\M_n\smallsetminus \SS_1$ act trivially in~\eqref{eq_act1}. 
To emphasize that its action depends on the coset $\SS$, we denote by 
$\tr(X|_\SS \wT_n)$ the trace of $\wT_n$ (that is of $\wT_{\SS_1}$) on any 
subspace $X\subset \W_w^{\G,\chi}$ preserved by it. 

\begin{prop} \label{L2.1} 
Let $\G\subset\G_1$ be a finite index subgroup, $k=w+2\ge 2$ an integer,
$\chi$ a character of $\G$ with kernel of finite index in $\G$, and
$\SS\subset \M_n$ a double coset satisfying Assumption~\ref{eq_star}. 
For any $\wT_n \in \RR_n$ satisfying~\eqref{Z1}, we have 
  $$\tr(W_w^{\G,\chi}|_\SS \wT_n)=\tr([\SS], M_k(\G,\chi)+ S_k^c(\G,\chi))\;.$$ 
\end{prop}
\begin{proof} By \eqref{2.5} and the Eichler-Shimura isomorphism
combined with the Shapiro lemma we have
\be \label{8}
\tr(W_w^{\G,\chi} |_\SS \wT_n)= \tr([\SS], S_{w+2}(\G,\chi)+S_{w+2}^c(\G,\chi))+\tr( C_w^{\G,\chi} |_\SS \wT_n)\;,
\ee
where $C_w^{\G,\chi}$ is the coboundary subspace defined by~\eqref{es}.
Therefore it is enough to show that 
$\tr(C_w^{\G,\chi}|_\SS \wT_n)=\tr([\SS], E_k(\G,\chi))$, where 
$E_k(\G,\chi)\subset M_k(\G,\chi)$ is the Eisenstein subspace. This can be 
shown by computing explicitly the left side, using Remark~\ref{r3.2}, 
and by comparing the result with the formula for the Eisenstein trace, which
is easily computed. For brevity we omit this computation,
and refer to our arXiv preprint 1408.4998v2 for the details.

A more conceptual proof is provided by the theory of modular symbols of Ash and 
Stevens. The space $W_w^{\G,\chi}$ is isomorphic with the space of modular 
symbols $\Symb_\G(V_w^\chi)$ defined in \cite[Sec. 4]{AS}, and the 
isomorphism is compatible with the action of Hecke operators on both sides. 
By \cite[Prop. 4.2]{AS}, we have a Hecke equivariant isomorphism between 
$\Symb_\G(V_w^\chi)$ and the compactly supported cohomology group 
$H^1_c(X_\G , \widetilde{V}_w^\chi)$ of the local system~$\widetilde{V}_w^\chi$
associated to $V_w^\chi$ on the modular surface $X_\G=\G\backslash\H$. Therefore 
we have a Hecke-equivariant isomorphism
$W_w^{\G,\chi} \simeq H^1_c(X_\G, \widetilde{V}_w^\chi),$ 
and since the latter space is Hecke isomorphic with $M_k(\G,\chi)+S_k^c(\G,\chi)$ by a version 
of the Eichler-Shimura isomorphism, the conclusion follows. 
\end{proof}

\section{Proof of Theorem~\ref{TF}}\label{s2.1}
First we show that the second version of Theorem~\ref{TF} is equivalent 
to the first. 
For $\G$ a finite index subgroup of $\G_1$, let $[M]_\G$ denote the 
$\G$-conjugacy class in $\PGL_2^+(\R)$ of the projection of a matrix 
$M\in \GL_2^+(\R)$. For fixed $M\in\SS$, the subsum in~\eqref{TFF} 
over $\G$-conjugacy classes $\X\subset [M]_{\G_1}$ equals
\be\label{1.5}
\sum_{\substack{A\in \ov{\G}\backslash\ov{\G}_1\\ 
 \pm AMA^{-1}\in \SS}} p_{k-2}(\tr M, \det M) (\pm 1)^k\wc(\pm AMA^{-1}) 
\frac{\e_\G([AMA^{-1}]_\G)}{[\stab_{\ov{\G}_1} M : 
\stab_{\ov{A^{-1}\G A}} M]}\;,
\ee  
where the same sign is chosen in all three places, and if $-1\notin\G$ 
at most one choice of signs is possible for each $A$.\footnote{The assumption 
on $\SS$ implies that if $-1\notin\G$ then $\SS\cap -\SS=\emptyset$. See the 
remark following Lemma~\ref{L1}.} Indeed, the~$\G$-conjugacy classes contained in $[M]_{\G_1}$ are $[AMA^{-1}]_\G$, 
with $A$ running through a set of representatives for $\ov{\G}\backslash\ov{\G}_1$
such that $\pm AMA^{-1}\in \SS$; for fixed such $A$ and varying $h\in  \stab_{\G_1} M$, 
elements in the cosets~$\G A h$ give the same conjugacy class $[AMA^{-1}]_\G$, 
and the number of such distinct cosets is easily seen to equal the index 
in the denominator above. When $\stab_{\ov{\G}_1} M$ is finite, 
the fraction in the sum above equals $\e([M]_{\G_1})$, 
since  $|\stab_{\ov{A^{-1}\G A}} M|=|\stab_{\ov{\G}} AMA^{-1}|$.
Therefore formula~\eqref{TFF} implies~\eqref{TF2}, and since the reasoning 
above is reversible, the two trace formulas are equivalent. 

We now apply Theorem~\ref{Tcoh} to the $\G_1$-module $V_w^{\G,\chi}$ 
to prove the second version of Theorem~\ref{T1}. The module 
$V_w^{\G,\chi}$ admits a a $\G_1$-equivariant pairing, given by
  \[\lla P, Q\rra:=\frac{1}{[\G_1:\G]} 
  \sum_{A\in \G\backslash\G_1} \la P(A), \ov{Q(A)} \ra \;.\]
where $\la (ax+b)^w, (cx+d)^w\ra=(ad-bc)^w\;$ is the well-known 
$\SL_2(\R)$-invariant pairing on $V_w$. It is clear that the definition is independent of the system of representatives chosen in the summation, 
and this pairing is nondegenerate and $\G_1$-invariant, so the hypothesis of  
Theorem~\ref{Tcoh} is satisfied.

The space of $\G_1$-invariants $(V_w^{\G,\chi})^{\G_1}$ is trivial if 
$(w,\chi)\ne (0, {\bf 1})$, and if $w=0$ and $\chi={\bf 1}$, it is one-dimensional spanned by the constant polynomial $P_0$, with
$P_0(A)=1$ for $A\in \G_1$. In the latter case we have
\be\label{3.7}
\tr( (V_0^{\G,{\bf 1}})^{\G_1}|_\SS T_n^\infty )= P_0|_\SS T_{\SS}^\infty (I)
 =\sum_{M\in \G_1\backslash\G_1 \SS }  \wc(M_I) , \ee
where $M$ runs through a system of representatives for $\G_1\backslash\G_1 \SS$
and $M_I\in \SS$ is any element such that $M\in \G_1 M_I$. By Assumption~\ref{eq_star}, 
$M_I$ runs over a system of representatives for $\G\backslash\SS$, so the 
last sum equals $\sum_{\ss\in \G\backslash\SS} \wc(\ss)$. 

By Theorem~\ref{Tcoh} and Proposition~\ref{L2.1}, the second version of Theorem~\ref{TF} is proved once 
we compute below the trace of $M\in \M_n$ on the module  $V_w^{\G,\chi}$. \qed

\begin{lemma}\label{L1}Assume that $\chi(-1)=(-1)^k$ if $-1\in\G$. For any $M\in \M$ we have
 \[\tr(V_w^{\G,\chi}|_{\SS} M)=p_w (\tr M,\det M)\cdot \cc_{\G,\SS}^\chi(M) \;,\]
where 
  $$\cc_{\G,\SS}^\chi(M)=\sum_{\substack{A\in \ov{\G}\backslash\ov{\G}_1\\ \pm AMA^{-1}\in \SS}}
  (\pm 1)^w\wc(\pm AMA^{-1}) \;. $$
\end{lemma}
\begin{remark*} If $-1\in \G$ the signs can be chosen arbitrarily.  
If $-1\notin\G$,  Assumption~\ref{eq_star} implies that $M$ and $-M$ cannot both belong to $\SS$, so in each term 
at most one choice of signs is possible. 
\end{remark*}
\begin{proof}Let $C_\G$ be a system of representatives for $\G\backslash\G_1/\{\pm 1\}$.
We have a decomposition 
\[V_w^{\G,\chi}=\bigoplus_{A\in C_\G} V_w^{(A)}\;,
\]
where $V_w^{(A)}\simeq V_w$ is the space of $P\in V_w^{\G,\chi}$ with $P(B)=0$ 
if $A\ne B\in C_\G$. If $MA^{-1}\not\in \G_1\SS$, then $M$ maps $V_w^{(A)}$ into 
$\oplus_{B\ne A} V_w^{(B)}$; if $MA^{-1}\in \G_1\SS$, there are unique 
$A_M\in C_\G$, $M_A\in\SS$ such that $MA^{-1}=\pm A_M^{-1} M_A$ (the sign can be 
assumed $+1$ if $-1\in\G$), and 
$$P|_\SS M(A)=\wc(M_A)(\pm 1)^w P( A_M)|_{-w} M.$$ 
It follows that the space $V_w^{(A)}$ contributes to the trace only if $A_M=A$, that is $\pm AMA^{-1}\in \SS$, and its contribution is 
$(\pm 1)^w \wc(\pm AM A^{-1}) \tr(V_w|_{-w} M). $
The conclusion follows from the fact that the last trace is $p_w(\tr M,\det M)$. 
\end{proof}

\section{Proof of Theorem~\ref{T3}}\label{sec5}
Since $(\ell,n)=1$ we easily check that $\SS=\G_n\ss\G_n$ satisfies 
Assumption~\ref{eq_star}, and 
$$|\G_n\bsh\SS|=|\G_0(\ell)\bsh\G_1|=\varphi_1(\ell).$$ 

In the trace formula~\eqref{TF2}, the conjugacy classes 
$X\subset \ov{\G_1 \ss \G_1}$ with $\e (X)\ne 0$ have representatives 
$M_X\in M_\ell$  such that $|\tr M_X|\le l+1$. 
The condition $\pm AM_X A^{-1}\in \G_n\ss\G_n$ implies that 
$\pm\tr M_X\equiv \tr\ss=\ell+1 \pmod n$, and  from $n>2\ell+2$ we conclude that 
the conjugacy classes that contribute to the formula have a representative with
$\tr M_X =\tr \ss =l+1$. Trace formula~\eqref{TF2} becomes:
\[
\tr ([\G_n\ss\G_n], M_k(\G_n)+S_k(\G_n) )= 
p_{k-2}(\ell+1,\ell)\cdot\sum_{\substack{X\subset M_\ell\\ \tr M_X=\ell+1}}C_{n,\ss}(M_X) 
+\delta_{k,2}\varphi_1(\ell),
\]
where $C_{n,\ss}(M_X)= \#\{A\in\ov{\G}_n\bsh \ov{\G}_1 : 
A M_X A^{-1}\equiv \ss\! \pmod{n} \} $.
Here we used the fact that $(\ell,n)=1$ to conclude 
$\G_n\ss\G_n=\{M\in M_\ell : M\equiv \ss \pmod{n}\}$~\cite[Lemma 3.29]{Sh}. 

As representatives for $\G_1$-conjugacy classes in 
$M_\ell$ of trace $\ell+1$ we take the matrices $\sm \ell b01$ with 
$0\le b<\ell-1$, so we can assume $M_X=\sm \ell b01$. The 
equation $A M_X \equiv \ss A\! \pmod{n}$ for $A=\sm xyzt
\in \SL_2(\Z)$ is equivalent to
\[ b z\equiv 0 ,\ \ (\ell-1)z\equiv 0 ,\ \ (\ell-1)y\equiv-bt \pmod{n}.  
\]
Write $n=dn'$, $\ell-1=d s$ with $(n',s)=1$. From $d|bz$, $d|bt$ it follows 
that $d|b$, so 
\[z\equiv 0 \pmod {n'},  \ \ sy\equiv -tb/d \pmod{n'} . \]
It follows that $xt\equiv 1 \pmod {n'}$, and for every such choice of $x,t$, 
there is a unique solution $A\in \SL_2(\Z/n'\Z)$. Each such solution has
$\#\SL_2(\Z/n\Z)/\#\SL_2(\Z/n'\Z) $ lifts to  $\SL_2(\Z/n\Z)$, and since
$\#\SL_2(\Z/n\Z)/\{\pm 1\}=[\ov{\G}_1:\ov{\G}_n ]=\varphi_2(n)$  we obtain
\[ C_{n,\ss}(M_X)= \frac{\varphi(n')}{2} \cdot\frac{\varphi_2(n)}{\varphi_2(n')}, \]
independent of $b$. Since
$p_{k-2}(\ell+1,\ell)=\frac{\ell^{k-1}-1}{\ell-1}$ and there are $\ell-1$ conjugacy
classes in the sum over $X$ above, the claim follows.

%\end{proof}

\end{document}